\documentclass[letterpaper,11pt]{amsart}

\usepackage{amsfonts,amsthm,latexsym,amsmath,amssymb,amscd,amsmath, mathrsfs, epsf, xypic,geometry}
\usepackage{color}
\usepackage[utf8]{inputenc}
\usepackage{picture}
\usepackage{tikz}
\usepackage{graphicx}
\usepackage{verbatim}

 \newtheorem{theorem}{Theorem}[section]
 \newtheorem{corollary}[theorem]{Corollary}
 \newtheorem{lemma}[theorem]{Lemma}
 \newtheorem{proposition}[theorem]{Proposition}
 
 \theoremstyle{definition}
 
 \theoremstyle{remark}
 \newtheorem{remark}[theorem]{Remark}
 \theoremstyle{definition}
 \newtheorem{example}[theorem]{Example}

\newcommand{\rk}{\mathrm{rk}}

\newcommand{\RR}{\mathbb{R}}
\newcommand{\CC}{\mathbb{C}}
\newcommand{\XX}{\mathbb{X}}
\newcommand{\PP}{\mathbb{P}}

\newcommand{\KK}{\mathbb{K}}

\title{On the Real rank of monomials}

\author[E. Carlini]{Enrico Carlini}
\address[E. Carlini]{DISMA- Department of Mathematical Sciences, Politecnico di Torino, Turin, Italy}
\email{enrico.carlini@polito.it}

\author[M. Kummer]{Mario Kummer}
\address[M. Kummer]{Fachbereich Mathematik und Statistik, 
Universit\"{a}t Konstanz, Constance, Germany}
\email{mario.kummer@uni-konstanz.de}

\author[A. Oneto]{Alessandro Oneto}
\address[A. Oneto]{Department of Mathematics, Stockholm University, SE-106 91, Stockholm,  Sweden}
\email{oneto@math.su.se}

\author[E. Ventura]{Emanuele Ventura}
\address[E. Ventura]{Department of Mathematics and Systems Analysis, Aalto University,  Finland}
\email{emanuele.ventura@aalto.fi}


\begin{document}

\maketitle

\date{}

\begin{abstract}

In this paper we study the real rank of monomials and we give an upper bound for the real rank of all monomials. We show that the real and the complex ranks of a monomial coincide if and only if the least exponent is equal to one.

\end{abstract}


\section{Introduction}

Let $\mathbb K$ be a field and $S=\mathbb K[x_0,\ldots,x_n] = \bigoplus_{d\geq 0} S_d$ be the ring of polynomials with
coefficients in $\mathbb K$ and with standard grading, i.e. $S_d$ is the $\KK$-vector space of homogeneous polynomials,
or forms, of degree $d$.

Given $F\in S_d$, we define a {\it Waring decomposition of $F$ over $\KK$} as a sum
\begin{equation}\label{Waring dec}
F=\sum_{i=1}^s c_iL_i^d,
\end{equation}
where  $c_i \in \KK$ and the $L_i$'s are linear forms over $\KK$. The smallest $s$ for which such a decomposition exists is called {\it Waring rank of $F$ over $\KK$} and it is denoted by $\rk_{\KK}(F)$.


The study of Waring decompositions over the complex numbers goes back to the work of Sylvester \cite{Sylv} and other
geometers and algebraists of the XIX century; see  \cite{IK} for historical details. Even if it has a long history, it was only in 1995 that
the Waring ranks were determined for {\it general} forms over the complex numbers;  see \cite{AH}.

However, the computation of the Waring rank is not known for all forms. The case of complex
binary forms goes back to Sylvester and has been recently reviewed in \cite{CS}. More recently, some progress has been made:
the complex Waring rank of monomials (and sums of pairwise coprime monomials) and
complex ranks of other sporadic families of polynomials have been determined in \cite{CCG} and \cite{CCCGW} respectively.
Some algorithms have been proposed, but they require technical restrictions to compute the rank; see
\cite{BGI, BCMT, CGLM, OO}.




The computation of Waring ranks over the real numbers is even more difficult. For instance, the real rank for monomials is only known
in the case of two variables; see \cite{BCG}. In \cite{Rez}, results on the rank of binary forms over the reals and other
fields are exhibited.

In the present paper we study the connection between the complex and real rank of monomials.


\smallskip

The paper is structured as follows. In Section \ref{Background}, we introduce notation and background. In Section \ref{Real and complex of monomials}, we prove our main results. We give an upper bound for the real rank of
monomials in Theorem \ref{Upper Bound THM}. In Proposition \ref{Prop_d=2_upperBound} we show that this upper bound is not always sharp. In Theorem \ref{characterization}, we prove that for monomials the real and complex rank coincide if and only if their least exponent is equal to one.

\section{Background}\label{Background}

Let $T=\mathbb K[X_0,\ldots,X_n] = \bigoplus_{d\geq 0} T_d$ be the dual ring of $S$ acting by differentiation on $S$:
$$
X_i \circ x_j :=\frac{\partial}{\partial x_i}x_j.
$$
For any homogeneous polynomial $F\in S$, the {\it perp ideal} of $F$ is
$$
F^{\perp}:=\lbrace \partial\in T  ~|~  \partial F=0 \rbrace \subset T.
$$
One of the key results to study the Waring problem is the Apolarity Lemma.

\begin{lemma}[{\sc Apolarity Lemma}, {\cite[{\bf Lemma} 1.15]{IK}}]\label{Apolarity}

Let $F\in S$ be a form of degree $d$. The following are equivalent:
\begin{enumerate}
 \item $F = \sum_{i=1}^r L_i^d$, where the $L_i$'s are linear forms;
 \item $ I_{\XX}\subset F^{\perp}$, where $I_{\XX}$ is the ideal defining a set $\XX$ of $r$ reduced points.
\end{enumerate}
\end{lemma}

\noindent A set of reduced points $\XX$ in $\PP^n$ is said to be {\it apolar} to $F$ if $I_{\XX}\subset F^{\perp}$.

As already mentioned, the complex rank of monomials has been determined in \cite{CCG}.

\begin{theorem}[{\cite[{\bf Corollary} 3.3]{CCG}}]\label{CCG Theorem}

Let $M=x_0^{a_0}x_1^{a_1}\ldots x_n^{a_n}$ with $0<a_0\leq a_1\leq \ldots \leq a_n$. Then $$
\rk_{\CC} (M)=\frac{1}{a_0+1}\prod_{i=0}^n(a_i+1).
$$

\end{theorem}

Moreover, any set of reduced points apolar to a monomial $M$, whose cardinality is equal to $\rk_{\CC}(M)$, is
a complete intersection.

\begin{theorem}[{\cite[{\bf Theorem} 1]{BBT}}]\label{BBT Prop}

Let $M=x_0^{a_0}x_1^{a_1}\ldots x_n^{a_n}$ with $0<a_0\leq a_1\leq \ldots \leq a_n$. Then, any ideal $I_{\XX}$ of a set of $\rk_{\CC}(M)$ points apolar to $M$ is a complete intersection of the form
\[(X_1^{a_1+1}-F_1X_0^{a_0+1},\ldots,X_n^{a_1+1}-F_nX_0^{a_0+1}),\]
where the forms $F_i$'s have degrees $a_i-a_0$.

\end{theorem}

The real rank of monomials in two variables has been computed in \cite{BCG}.

\begin{proposition}[{\cite[{\bf Proposition} 4.4]{BCG}}]\label{BCG Prop}
If $M = x_0^{a_0}x_1^{a_1}$, then $\rk_{\RR}(M) = a_0+a_1$.
\end{proposition}

In view of these results, we can easily note that whenever one of the exponents of a binary monomial is equal to
one, then real and complex rank coincide. Hence, in the case of two variables the monomials whose real and complex rank coincide 
are those whose least exponent is one. The proof of Proposition \ref{BCG Prop} is a combination of Apolarity Lemma with a straightforward application of the Descartes' rule of signs. \\

\begin{lemma}[{\cite[{\bf Lemma} 4.1-- 4.2]{BCG}}]\label{BCG Lemma}

Let $F$ be a real polynomial in one variable $$F = x^d + c_1x^{d-1}+ \ldots + c_{d-1}x + c_d.$$
Then
\begin{enumerate}
 \item for any $i<d$, there is a choice of $c_j$'s such that $F$ has distinct real roots and $c_i=0$;
\item if $c_i=c_{i+1}=0$ for some $0 < i < d$, then $F$ has a non real root.
\end{enumerate}

\end{lemma}

\indent In order to prove the characterization of monomials whose complex and real ranks are equal, featured in Theorem \ref{characterization}, we recall the trace bilinear form of finite $\mathbb K$-algebras; we refer to \cite{PRS} for details. \\
\indent Let $A=\mathbb K[X_0,\ldots,X_n]/I$ be a finite $\mathbb K$-algebra. For any element $F\in A$, we define the endomorphism $m_F\in \textnormal{End}(A)$ to be the multiplication by $F$. Since $A$ is a finite dimensional $\mathbb K$-vector space, we have a trace map $\textnormal{Tr}_{A/\mathbb K}: \textnormal{End}(A)\rightarrow \mathbb K$, which is the trace of the corresponding matrix. We define a symmetric bilinear form 
$$B(F,G): A \otimes A \rightarrow \mathbb K,$$ \noindent by 
$$
B(F,G)=\textnormal{Tr}(m_F\cdot m_G)=\textnormal{Tr}_{A/\mathbb K}(m_{F\cdot G}).
$$

The following result is featured in \cite[Thm. 2.1]{PRS}; we give an elementary proof for the sake of completeness.

\begin{proposition}\label{bilinear form positive definite}

Let $A$ be a reduced finite $\RR$-algebra of dimension $N$. If $\textnormal{ Spec}A$ consists only of $\RR$-points, then the bilinear form $$B: A\otimes A\to \RR,\,(F,G)\mapsto \textnormal{Tr}_{A/\RR}(m_{F\cdot G})$$ is positive definite.
 
\end{proposition}

\begin{proof}

The $\RR$-algebra $A$ is isomorphic to $\RR\times\cdots\times \RR$ because $A$ is reduced. The representing matrix of the $\RR$-linear map $A\to A$ given by multiplication with an element $F=(F_1,\ldots,F_N)\in A$ is the diagonal matrix with diagonal entries $F_1,\ldots,F_N$. Thus, we have $B(F,F)=\textnormal{Tr}_{A/\RR}(m_{F^2})=F_1^2+\ldots+F_N^2\geq 0$ and $B(F,F)=0$ if and only if $F=0$. Hence $B$ is positive definite. 
\end{proof}

\section{Real and complex ranks of monomials}\label{Real and complex of monomials}
\indent

In this section we prove our main results. First, we give an upper bound for the real rank of monomials.

\begin{theorem}\label{Upper Bound THM}
 If $M=x_0^{a_0}\ldots x_n^{a_n}$ with $0< a_0\leq \ldots\leq a_n$, then $$\rk_{\RR}(M)\leq \frac{1}{2a_0}\prod_{i=0}^n (a_i+a_0).$$
\end{theorem}

\begin{proof}
 Clearly $M^{\perp}=(X_0^{a_0+1},\ldots ,X_n^{a_n+1})$. Let us consider
 \begin{align*}
  G_i & = X_0^{a_0+1}g_{i}(X_0,X_i)+X_i^{a_i+1}h_{i}(X_0,X_i),
 \end{align*}
 where $\deg g_i=a_i-1$ and $\deg h_i=a_0-1$, for every $i=1,\ldots,n$.

\noindent Each $G_i$ is a binary form of degree $a_0+a_i$ where the monomial $X_0^{a_0}X_i^{a_i}$ does not appear. Thus, by Lemma \ref{BCG Lemma}, there exists a choice of $g_i$ and $h_i$ such that $G_i$ has $a_0+a_i$ distinct real roots, say $p_{i,j}$ with $j=1,\ldots,a_0+a_i$. Therefore, the ideal $(G_1,\ldots,G_n)\subset M^{\perp}$ is the ideal of the following set of distinct real points:
 $$\XX = \Big \lbrace[1:p_{1,j_1}:\ldots:p_{n,j_n}] ~\Big|~ 1\leq j_i\leq a_0+a_i,~\text{ for }i=1,\ldots,n\Big \rbrace.$$
\noindent By the Apolarity Lemma, the proof is complete.
\end{proof}

As a direct corollary, we have that whenever the least exponent of a monomial is one, then
the real and the complex rank coincide.

\begin{corollary}\label{Exact real rank}

If $M=x_0x_1^{a_1}\ldots x_n^{a_n}$, then $\rk_{\CC}(M)=\rk_{\RR}(M)$.

\end{corollary}

\begin{proof}

If $a_0=1$, the upper bound given by Theorem \ref{Upper Bound THM} equals the complex rank of $M$ given by the formula
in Theorem \ref{CCG Theorem}.
\end{proof}

\begin{remark}\label{a0=1 decomposition}

Here we produce a  family of minimal real Waring decompositions for $M=x_0x_1^{a_1}\ldots x_n^{a_n}$.
Consider a set of real numbers
$$
\Big \lbrace p_{i,j} \in \RR \ \Big| \ i=1,\ldots,n \mbox{ and } j=1,\ldots,a_i+1\Big\rbrace,
$$
\noindent such that $$\sum_{j=1}^{a_i+1} p_{i,j}=0 \mbox {, for any } i,$$
\noindent and $p_{i,a}\neq p_{i,b}$ if $a\neq b$, for any $i$. Hence, we have a set $\XX$ of
$(a_1+1)\cdot \ldots \cdot (a_n+1)$ distinct reduced real points in $\PP^n$ given by
$$\XX = \Big \lbrace[1:p_{1,j_1}:\ldots:p_{n,j_n}] ~\Big|~ 1\leq j_i\leq a_i+1,~ i=1,\ldots,n\Big\rbrace.$$
We define the forms $G_i=\prod_{j=1}^{a_i+1}(X_i-p_{i,j}X_0), \mbox{ for } i=1,\ldots,n$.
Since $\sum_{j=1}^{a_i+1} p_{i,j}=0$, the $G_i$'s are in $M^{\perp}$ and they generate the  ideal of $\XX$,
which, by the Apolarity Lemma, gives a minimal real Waring decomposition of $M$. \\
\indent The family of decompositions shown above can be parametrized as follows. Each decomposition is in bijection
with $n$ binary forms of degree $a_i$ (because the last zero is determined by the others) whose roots are all real and distinct. Each binary form of degree $a_i$ is in the projective space
$\mathbb P(\RR[X_0,X_i]_{a_i})$ and sits in the complement of the discriminant. Furthermore, each of these binary forms is in the connected component consisting of binary forms whose roots are all real, i.e. hyperbolic binary forms. Thus, the $n$-fold product of the connected components consisting of hyperbolic binary forms of degrees $a_i$  gives the desired parametrization.

\end{remark}

We give a result on the number of real solutions of some family of complete intersections, which has a similar flavour of the Descartes' rule of signs in the context of systems of polynomial equations. 

\begin{theorem}\label{generalization of Descartes}

Let $2\leq a_0 \leq \ldots \leq a_n$. For $i=1,\ldots,n,$ let $F_i \in \RR[X_1,\ldots,X_n]$ be a polynomial of degree at most $a_i-a_0$. Then the system of polynomial equations defined by 

 \begin{eqnarray}\label{system with gaps}                                                                    
  G_1=X_1^{a_1+1}+F_1&=&0,  \nonumber \\
 & \vdots& \\
 G_n=X_n^{a_n+1}+F_n&=&0, \nonumber
  \end{eqnarray}

 does not have $\prod_{i=1}^n(a_i+1)$ real distinct solutions. \\
 
 \begin{proof}
 
We give a proof by contradiction, assuming that the number of real distinct solutions is $\prod_{i=1}^n(a_i+1)$. 
Let $I=(G_1,\ldots,G_n)\subseteq \RR[X_1,\ldots,X_n]$ be the ideal generated by $G_1,\ldots,G_n$. We consider the $\RR$-algebra $A=\RR[x_1,\ldots,x_n]/I$ and the bilinear form \[B: A\otimes A\to \RR,\,(H,K)\mapsto \textnormal{Tr}_{A/\RR}(m_{H\cdot K}).\] \indent Since the system has $\prod_{i=1}^n(a_i+1)$ real distinct solutions, $B$ is positive definite by Proposition \ref{bilinear form positive definite}. The residue classes of the monomials $X_1^{\alpha_1}\cdots X_n^{\alpha_n}$ with $0\leq \alpha_i\leq a_i$ form a basis of $A$ as a vector space over $\RR$. We want to show that the representing matrix $M$ of the $\RR$-linear map \[\varphi: A \to A,\,H\mapsto X_1^2\cdot H\] with respect to this basis has only zeros on the diagonal. This would imply $B(X_1,X_1)=\textnormal{Tr}_{A/\RR}(m_{X_1^2})=0$, which in turn would imply that $B$ is not positive definite. \\
For $0\leq\alpha_i\leq a_i$ we have $\varphi(X_1^{\alpha_1}\cdots X_n^{\alpha_n})=X_1^{\alpha_1+2}\cdot X_2^{\alpha_2}\cdots X_n^{\alpha_n}$. If $\alpha_1+2\leq a_1$, then the column of $M$ corresponding to the basis element $X_1^{\alpha_1}\cdots X_n^{\alpha_n}$ has its only nonzero entry at the row corresponding to the basis element $X_1^{\alpha_1+2}\cdot X_2^{\alpha_2}\cdots X_n^{\alpha_n}$. If $\alpha_1+2>a_1$, then $\varphi(X_1^{\alpha_1}\cdots X_n^{\alpha_n})=-F_1\cdot X_1^{\alpha_1+1-a_1}\cdot X_2^{\alpha_2}\cdots X_n^{\alpha_n}$. It follows from our assumptions on the degrees of the $F_i$'s, that the element $\varphi(X_1^{\alpha_1}\cdots X_n^{\alpha_n})$ is in the span of all basis elements corresponding to monomials of degree smaller than $\sum_{i=1}^n \alpha_i$. In both cases, the corresponding diagonal entry of $M$ is zero. This concludes the proof. 
 
\end{proof}

\end{theorem}

We now give a characterization for those monomials whose real and complex ranks coincide.

\begin{theorem}\label{characterization}

 Let $M = x_0^{a_0}\cdots x_n^{a_n}$ be a degree $d$ monomial with $0<a_0\leq \ldots \leq a_n$. Then
 $$
  \rk_{\RR}(M) = \rk_{\CC}(M)
 $$
if and only if $a_0=1$.

\begin{proof}

If $a_0=1$, Corollary \ref{Exact real rank} proves the statement. Suppose that  $\rk_{\RR}(M) = \rk_{\CC}(M)$ and let $\XX$ be a minimal set of real points apolar to $M$. Assume by contradiction that $a_0\geq 2$. By Theorem \ref{BBT Prop}, we know that $\XX$ is a complete intersection and  we may dehomogenize by $X_0 = 1$. The set $\XX$ gives the solutions to a system of polynomial equations of the form (\ref{system with gaps}) in Theorem \ref{generalization of Descartes}. This is a contradiction and this concludes the proof.

\end{proof}
\end{theorem}

Finally, we show that the upper bound in Proposition \ref{Upper Bound THM} is not always sharp.

\begin{proposition}\label{Prop_d=2_upperBound}
Let $M=x_0^2\ldots x_n^2$. Then $\rk_{\RR}(M) \leq (3^{n+1}-1)/2$.
\end{proposition}

\begin{proof}
We explicitly give an apolar set of points for $M$ as follows.
For any $i=0,\ldots,n$, let us consider the set  $$\XX_i = \Big \lbrace[p_0:\ldots:p_{i-1}:1:p_{i+1}:\ldots:p_n]\in\PP^n ~|~ p_i\in\{0,\pm1\}\Big \rbrace.$$
We can easily determine the cardinality of $\XX = \bigcup_{i=0}^{n}\XX_i$. From all $(n+1)$-tuples
$(p_0,\ldots,p_n)$ with $p_i=0,\pm 1$, we need to discard $(0,\ldots,0)$, since it does not correspond to any
point in the projective space. We are double counting, since $(p_0,\ldots,p_n)$ and $(-p_0,\ldots,-p_n)$ define
the same point in the projective space. Thus, $|\XX| = (3^{n+1}-1)/2$.
For each  $P\in \XX$,  let $L_P$ denote the corresponding linear form $p_0x_0+\ldots+p_nx_n$ and $n(P)$ the number
of entries different from zero. For each $i=1,\ldots,n+1$, we set $$R_i = \sum_{\substack{P\in\XX \\ n(P)=i}} L_P^{2n+2}.$$
\noindent By direct computation, we obtain $$\frac{(2n+2)!}{2}x_0^2\ldots x_n^2 = \sum_{i=1}^{n+1} (-2)^{n+1-i}R_i.$$
\noindent Thus, $\mathbb X$ is apolar to $M$ and this concludes the proof.
\end{proof}

\begin{example}

For $n=1$, we have the following real decomposition of $M=x_0^2x_1^2$:
$$12 x_0^2x_1^2 = R_2 -2 R_1 =  (x_0+x_1)^4+(x_0-x_1)^4 - 2(x_0^4+x_1^4).$$
\noindent For $n=2$, we have $\rk_{\CC}(x_0^2x_1^2x_2^2)=9\text{ and }10\leq\rk_{\RR}(x_0^2x_1^2x_2^2)\leq 13$:

$$ 360 x_0^2x_1^2x_2^2  = R_3 - 2 R_2 + 4 R_1 = $$
$$ = (x_0+x_1+x_2)^6 + (x_0+x_1-x_2)^6 + (x_0-x_1+x_2)^6 + (x_0-x_1-x_2)^6 + $$
$$ -2[ (x_0+x_1)^6 +(x_0-x_1)^6 +(x_0+x_2)^6+(x_0-x_2)^6 + (x_1+x_2)^6 +(x_1-x_2)^6] + $$
$$ +4 (x_0^6+x_1^6+x_2^6). $$

\noindent In \cite[Example 6.7]{MMSV}, it is proved that $\rk_{\RR}(x_0^2x_1^2x_2^2) > 10$. 
\end{example}

\smallskip

{\bf Acknowledgments.} We would like to thank A. Di Scala, C. Guo and G. Ottaviani for useful discussions. 
The third and the fourth author would like to thank the School of Mathematical Sciences at Monash University for
the hospitality when this project started.
The first author was supported by GNSAGA group of INDAM.
The second author was supported by the Studienstiftung des deutschen Volkes.
The third author was partially supported by G S Magnuson Fund from Kungliga Vetenskapsakademien.
The fourth author was supported from the Department of Mathematics and Systems Analysis of
Aalto University.

\bigskip

\end{document}